\documentclass[12pt,reqno]{amsart}

\usepackage{amssymb, amscd, stmaryrd}
\usepackage{enumerate}
\usepackage[dvips,dvipdf]{graphicx}
\usepackage[usenames,dvipsnames]{color}

\usepackage{color}
\usepackage{color, colortbl}
\usepackage{multirow}
\usepackage[table]{xcolor}
\usepackage{comment}
\usepackage{subfigure}
\usepackage{epstopdf}
\usepackage{tikz}

\usepackage{geometry}
\makeatletter
\newcommand\figcaption{\def\@captype{figure}\caption}
\newcommand\tabcaption{\def\@captype{table}\caption}
\makeatother

\def\tb{\textcolor{black}}

\newcommand{\pa}{\partial}

\newcommand{\maT}{\mathcal T}

\newcommand{\be}{\begin{eqnarray}}
\newcommand{\ee}{\end{eqnarray}}
\newcommand{\ben}{\begin{eqnarray*}}
\newcommand{\een}{\end{eqnarray*}}

\def\be{\begin{equation}}
\def\ee{\end{equation}}
\def\bes{\begin{equation*}}
\def\ees{\end{equation*}}
\def\beq{\begin{eqnarray}}
\def\eeq{\end{eqnarray}}
\def\beqs{\begin{eqnarray*}}
\def\eeqs{\end{eqnarray*}}
\def\bal{\begin{aligned}}
\def\eal{\end{aligned}}
\def\bsqs{\begin{subequations}}
\def\esqs{\end{subequations}}

\newcommand{\maK}{\mathcal K}

\newtheorem{theorem}{Theorem}[section]

\newtheorem{lemma}[theorem]{Lemma}
\newtheorem{lem}[theorem]{Lemma}

\theoremstyle{definition}
\newtheorem{definition}[theorem]{Definition}
\theoremstyle{definition}
\newtheorem{algorithm}[theorem]{Algorithm}
\theoremstyle{remark}
\newtheorem{remark}[theorem]{Remark}
\newtheorem{example}[theorem]{Example}

\author[H. Li, P. Yin, Z. Zhang]{Hengguang Li$^\dagger$, Peimeng Yin$^\dagger$ and Zhimin Zhang$^\dagger$$^\ddagger$}
\address{$^\dagger$ Wayne State University, Department of Mathematics, Detroit, MI 48202, USA}
\email{li@wayne.edu; pyin@wayne.edu; zzhang@math.wayne.edu;}
\address{$^\ddagger$ Beijing Computational Science Research Center, Beijing 100193, China} \email{zmzhang@csrc.ac.cn}

\begin{document}
\title[FEM for biharmonic equation]{A $C^0$ finite element method for the biharmonic problem with Navier boundary conditions in a polygonal domain}

\date{\today}

\begin{abstract}
In this paper, we study the biharmonic equation with the Navier boundary conditions in a polygonal domain. In particular, we propose a method that effectively decouples the 4th-order problem into a system of Poisson equations. Different from the usual mixed method that leads to two Poisson problems but only applies to convex domains, the proposed decomposition involves a third Poisson equation to confine the solution in the correct function space, and therefore can be used in both convex and non-convex domains. A $C^0$ finite element algorithm is in turn proposed to solve the resulted system. In addition, we derive the optimal error estimates for the numerical solution on both quasi-uniform meshes and graded meshes. Numerical test results are presented to justify the theoretical findings.
\end{abstract}

\maketitle


\section{Introduction}
Let $\Omega\subset \mathbb R^2$ be a polygonal domain. Consider the biharmonic problem
\begin{eqnarray}\label{eqnbi}
\Delta^2u=f \quad {\rm{in}} \ \Omega,\qquad \quad u=0 \quad {\rm{and}} \quad \Delta u=0 \quad {\rm{on}} \ \pa\Omega.
\end{eqnarray}
The boundary conditions in (\ref{eqnbi}) are referred to as the homogeneous Navier boundary conditions \cite{MR1422248, MR2499118} that occur for example in the model for the static loading of a pure hinged thin plate. Equation (\ref{eqnbi}) is a 4th-order elliptic equation for which a direct finite element approximation usually involves  delicate constructions of the finite element space and of the variational formulation \cite{Ciarlet78, BS02}. An alternative approach is to use a mixed formulation to decompose the high-order problem into a system of  equations that may be easier to solve. This approach is particularly appealing for the biharmonic problem (\ref{eqnbi}) because the Navier boundary condition allows one to obtain two Poisson equations that are completely decoupled, which implies that a reasonable numerical solution should be achieved by merely applying a finite element Poisson solver in the mixed formulation.  However, it has been observed \cite{gerasimov2012, nazarov2007, MR2479119} that the performance of this usual mixed method depends on the domain geometry. In a convex domain, the corresponding numerical approximations converge to the solution of equation (\ref{eqnbi}) although the convergence rate may not be optimal.
 When the domain possesses reentrant corners, however, the result can be  misleading: this  mixed finite element formulation   produces numerical solutions that may be converging but to a wrong solution.


In this paper, we propose and analyze a  $C^0$  finite element method  for solving the biharmonic problem (\ref{eqnbi}). In particular, we shall devise an explicit mixed formulation to transform equation (\ref{eqnbi}) into a system of three Poisson equations. This is based on the observation that the aforementioned usual mixed formulation (decomposition into two Poisson equations) in fact defines a weak solution in a larger space than that for equation (\ref{eqnbi}). This mismatch in function spaces does not affect the solution in a convex domain; while in a non-convex domain, it allows additional singular functions  and therefore results in a solution different from that  in equation (\ref{eqnbi}). Our proposed mixed formulation ensures that the associated solution is identical to the solution of (\ref{eqnbi}) in both convex and non-convex domains. This is accomplished by  introducing an additional intermediate Poisson problem that confines the solution in the correct space.

To solve the proposed mixed formulation, we present a numerical algorithm based on the piecewise linear $C^0$ finite element.
Meanwhile, we carry out the error analysis on the finite element approximations for both the auxiliary function $w$ (see (\ref{eqnnew})) and   the solution $u$.  For the auxiliary function $w$, the error in the $H^1$ norm is standard  and has a convergence rate $h^{\frac{\pi}{\omega}}$ on a quasi-uniform mesh, where $\omega$ is the interior angle of the reentrant corner; its $L^2$ error estimate can be obtained using the duality argument.
For  the solution $u$, the error in the $H^1$ norm is  bounded by: (i) the interpolation error of the solution $u$ in $H^1$; (ii) the $L^2$ error for the auxiliary function $w$; and (iii) the $L^2$ error for the solution $\xi$ of the additional intermediate Poisson problem.
We shall show that the proposed algorithm has the optimal $H^1$ convergence rate for the solution $u$  on quasi-uniform meshes.

In addition, we derive regularity estimates for the proposed system in a class of Kondratiev-type weighted spaces.
Based on these regularity results, we in turn propose graded mesh refinement algorithms, such that the associated finite element methods recover the optimal convergence rate in the energy norm for  the auxiliary function $w$ even when $w$ is singular. To simplify the exposition and better present the idea, we adopt the  linear $C^0$ finite element method in this paper with the assumption that the domain $\Omega$ has at most one reentrant corner. The cases involving high-order finite elements and multiple reentrant corners will be discussed in a forthcoming paper.

The rest of the paper is organized as follows.  In Section \ref{sec-2},  we review the weak solutions of the biharmonic problem (\ref{eqnbi}) and the usual mixed formulation. In addition, we discuss the orthogonal space of the image of the operator $-\Delta$ and identify a basis function in this space. Then we propose a modified mixed formulation and show the equivalence of the solution to the original biharmonic problem.
In Section \ref{sec-3}, we propose the finite element algorithm and obtain  error estimates on quasi-uniform meshes for both the solution $u$ and the auxiliary function $w$.
In Section \ref{sec-4}, we introduce the weighted Sobolev space and derive the regularity estimates for the solution near the reentrant corner. Then we present the graded mesh algorithm and provide the optimal error estimates on graded meshes.
We report  numerical test results in Section \ref{sec-5} to validate the theory.

Throughout the paper,  the generic constant $C>0$ in our estimates  may be different at different occurrences. It will depend on the computational domain, but not on the functions  involved or the mesh level in the finite element algorithms.

\section{The biharmonic problem}\label{sec-2}

\subsection{Well-posedness of the solution} Denote by $H^m(\Omega)$, $m\geq 0$,  the Sobolev space that consists of functions whose $i$th derivatives are square integrable for $0\leq i\leq m$. Let $L^2(\Omega):=H^0(\Omega)$.
Recall that $H^1_0(\Omega)\subset H^1(\Omega)$ is the subspace consisting of functions with  zero trace on the boundary  $\pa\Omega$.   We shall denote the norm $\|\cdot\|_{L^2(\Omega)}$ by $\|\cdot\|$ when there is no ambiguity about the underlying domain.

The following variational formulation for equation (\ref{eqnbi}) can be obtained using integration by parts:
\begin{eqnarray}\label{eqn.firstbi}
a(u, v):=\int_\Omega \Delta u\Delta vdx=\int_\Omega fvdx=(f, v),\quad \forall v\in H^2(\Omega)\cap H^1_0(\Omega).
\end{eqnarray}
For a function $v\in H^2(\Omega)\cap H^1_0(\Omega)$, recall  the Poincar\'e-type inequality  \cite{Grisvard1992}: $\|\Delta v\|_{L^2(\Omega)}\geq C\|v\|_{H^2(\Omega)}$. Then by the Lax-Milgram Theorem, equation (\ref{eqn.firstbi}) defines a unique weak {solution} $u\in H^2(\Omega)\cap H^1_0(\Omega)$ for
 any  $f$ in the dual space of  $H^2(\Omega)\cap H^1_0(\Omega)$  (namely, $f\in \big(H^2(\Omega)\cap H^1_0(\Omega)\big)^*$). Meanwhile, the regularity of the solution $u$ depends on the given data $f$ and the domain geometry.

 \subsection{The usual mixed formulation}

 Intuitively, equation (\ref{eqnbi}) can be decoupled to the system of two Poisson problems by introducing an auxiliary function $w$ such that
\begin{eqnarray}\label{eqn7}
\left\{\begin{array}{ll}
-\Delta w=f \quad {\rm{in}} \ \Omega,\\
\hspace{0.6cm}w=0 \quad {\rm{on}} \ \pa\Omega;
\end{array}\right.
\qquad \qquad {\rm{and}}\qquad \qquad\left\{\begin{array}{ll}
-\Delta \bar{u}=w \quad {\rm{in}} \ \Omega,\\
\hspace{0.6cm}\bar{u}=0 \quad {\rm{on}} \ \pa\Omega.
\end{array}\right.
\end{eqnarray}
We refer to (\ref{eqn7}) as the usual mixed formulation. Note that numerical solvers for the Poisson problems (\ref{eqn7}) are readily available, while numerical approximation of the fourth-order problem (\ref{eqnbi}) is generally a much harder task.
The mixed weak formulation of (\ref{eqn7}) is to find $\bar{u}, w\in H_0^1(\Omega)$ such that
\begin{subequations}\label{weaksys}
\begin{align}
A( w, \phi ) = & ( f, \phi), \quad \forall \phi \in H_0^1(\Omega), \\
A( \bar{u}, \psi ) := & ( w, \psi), \quad \forall \psi \in H_0^1(\Omega),
\end{align}
\end{subequations}
where
$$
A(\phi,\psi)=\int_\Omega \nabla \phi \cdot \nabla \psi dx.
$$
Given $f\in H^{-1}(\Omega)\subset \big(H^2(\Omega)\cap H^1_0(\Omega)\big)^*$, it is clear that the weak solutions $\bar u, w$ are well defined by    (\ref{weaksys}) because they are solutions of decoupled Poisson problems \cite{Evans}. Since our goal is to solve the biharmonic problem (\ref{eqnbi}), an important question is whether the solution $u$ in (\ref{eqn.firstbi}) and the solution $\bar u$ in (\ref{weaksys}) are the same.

\begin{remark}\label{rk26} For $f\in H^{-1}(\Omega)$,   existing results suggest that under appropriate conditions, the solution $\bar u$ of the system (\ref{weaksys}) is equivalent to the solution $u$ of equation (\ref{eqn.firstbi}) in the sense that
\bes
u=\bar{u}\text{  in  } H^2(\Omega) \cap H_0^1(\Omega).
\ees
These conditions include
(i) the domain $\Omega$ and the given data $f$ being smooth, which can be verified by the regularity of these equations up to the domain boundary \cite{gilbarg1983, lions1972};
(ii) the polygonal domain $\Omega$ being convex  \cite{MR2479119}.
It is however pointed out that $u$ is not always equivalent to $\bar u$ when the polygonal domain $\Omega$ has reentrant corners, which is known as the Sapongyan paradox \cite{MR2270884,MR2479119}. In this case, the numerical solution for  (\ref{weaksys}) does not converge to the solution of the biharmonic problem (\ref{eqnbi}). In the next subsection, we shall study the structure of the solution in the presence of a reentrant corner in order to design effective numerical algorithms for equation (\ref{eqnbi}). \end{remark}

\subsection{Image of the Laplace operator and its orthogonal space} From now on, we assume the  given function  $f\in L^2(\Omega)$ in (\ref{eqnbi}). In addition,  assume that the polygonal domain $\Omega$ has a reentrant corner associated with the vertex $Q$ and the corresponding interior angle $\omega\in (\pi, 2\pi)$.
Without lost of generality, we  set $Q$ to be the origin. Let ($r, \theta$) be the polar coordinates centered at the vertex $Q$, such that $\omega$ is spanned by two half lines $\theta=0$ and $\theta=\omega$.
Given $R>0$,  we identify a sector $K_\omega^R \subset \Omega$ with radius $R$ as
$$
K_\omega^R = \{(r,\theta) |0 \leq r\leq R, 0\leq \theta \leq \omega\}.
$$
A sketch drawing of the domain $\Omega$ is shown in Figure \ref{fig:Omega}.

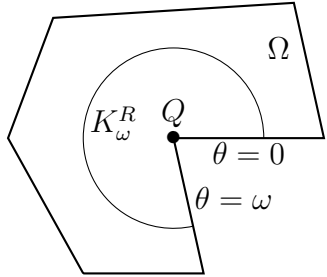
\begin{figure}
\begin{center}
\begin{tikzpicture}[scale=0.2]
\draw[thick]
(-6,-11) -- (2,-11) -- (0,-2) -- (10,-2) -- (8,7) -- (-8,6) -- (-11,-2) -- (-6,-11);
\draw (5,-3) node {$\theta = 0$};
\draw (4,-6) node {$\theta = \omega$};
\draw (7,4) node {$\Omega$};
\draw (6,-2) arc (0:283:6);
\draw (-4,-1) node {$K_\omega^R$};
\draw[thick] (0,-2) node {$\bullet$} node[anchor = south] {$Q$};
\end{tikzpicture}
\end{center}
\vspace*{-15pt}
    \caption{Domain $\Omega$ containing a reentrant corner.}
    \label{fig:Omega}
\end{figure}

The mapping $-\Delta: H^2(\Omega)\cap H^1_0(\Omega)\rightarrow L^2(\Omega)$ is injective and has a closed range \cite{Grisvard1992}. Denote by $\mathcal M$ the image of this mapping and by $\mathcal M^\perp$ its orthogonal complement. Then it follows $\mathcal M\oplus \mathcal M^\perp=L^2(\Omega)$. Therefore, if $w \in \mathcal M$ in (\ref{weaksys}), we have equivalent solutions $u=\bar u$ in $H^2(\Omega) \cap H_0^1(\Omega)$. When the domain $\Omega$ is convex, one has $\mathcal M=L^2(\Omega)$; namely, the solution of the Poisson equation with the Dirichlet boundary condition is always in $H^2(\Omega)$ when $f\in L^2(\Omega)$. Thus,  in a convex domain, the condition $w \in H^1_0(\Omega)\subset \mathcal M$ holds, and therefore the solutions $u$ and $\bar u$ are equivalent.   However, if   $\Omega$ contains reentrant corners, $\mathcal M$ is a strict subset of $L^2(\Omega)$ and in general $w \not\in \mathcal M$. Consequently, the solution in (\ref{eqn7}) $\bar u\notin H^2(\Omega)$ and it is different from the solution of (\ref{eqnbi}) $u\in H^2(\Omega) \cap H_0^1(\Omega)$. Fortunately, the space $\mathcal M^\perp$ is finite dimensional and it is possible to identify its basis.

We first introduce an $L^2$ function in domain $\Omega$ in the following way.
\begin{definition}\label{defs}
Given the parameters $\tau \in (0,1)$ and $R$ such that $K_\omega^R \subset \Omega$, we define an $L^2$ function in $\Omega$,
\begin{eqnarray}\label{ssol}
\xi(r, \theta; \tau, R):=s^-(r, \theta; \tau, R)+\zeta(r, \theta; \tau, R),
\end{eqnarray}
where
\begin{eqnarray}\label{sl2}
\bal
s^-(r, \theta; \tau, R)=&\eta(r; \tau, R)r^{-\frac{\pi}{\omega}}\sin\left(\frac{\pi}{\omega}\theta\right) \in L^2(\Omega),
\eal
\end{eqnarray}
with $\eta(r; \tau, R) \in C^\infty(\Omega)$ satisfying $\eta(r; \tau, R)=1$ for $0\leq r\leq \tau R$ and $\eta(r; \tau, R)=0$ for $r>R$, and
 $\zeta \in H_0^1(\Omega)$ satisfies
\be\label{l2part}
-\Delta \zeta = \Delta s^-  \text{ in } \Omega, \qquad \zeta = 0 \text{ on } \partial \Omega.
\ee
\end{definition}

From (\ref{sl2}), we see that $s^- \in C^\infty(\Omega \setminus K_\omega^\delta)$ for any $\delta>0$ and $s^- = 0$ for $(r, \theta) \in \Omega \setminus K_\omega^R$. Moreover, $\Delta s^- = 0$ if $r < \tau R$ or $r> R$.

\begin{lem}\label{slem}
For a given $\eta\in C^\infty(\Omega)$ as defined in Definition \ref{defs}, the function $\xi \in L^2(\Omega)$ is uniquely defined and satisfies
\be\label{seqn}
-\Delta \xi =0 \text{ in } \Omega, \qquad \xi =0 \text{ on } \partial \Omega.
\ee
Moreover, $\xi$ depends on the domain $\Omega$, but not on $\tau$ or $R$. Namely, for any $\tau_1, \tau_2$ and $R_1, R_2$ satisfying $0<\delta <\min\{ \tau_1 R_1, \tau_2 R_2 \}$, it follows
\be\label{paraindep}
\xi(r, \theta) :=\xi(r, \theta; \tau_1, R_1) = \xi(r, \theta; \tau_2, R_2).
\ee
\end{lem}
\begin{proof}
	{Note that for a given $\eta$, $s^-$ is uniquely defined in (\ref{sl2}) and therefore $\zeta\in H^1_0(\Omega)$ is uniquely defined via (\ref{l2part}). Thus,  $\xi$ is uniquely defined in (\ref{ssol}). We first prove (\ref{seqn}).}
 Taking $-\Delta$ on both side of (\ref{ssol}), we have by (\ref{l2part})
\bes
\bal
-\Delta \xi = -\left( \Delta s^- + \Delta \zeta \right) = 0.
\eal
\ees
It is clear that $\xi=0$ on $\partial \Omega$  since both $s^-$ and $\zeta$ equal zero on the boundary.

We proceed to prove (\ref{paraindep}). For $0<\delta <\min\{ \tau_1 R_1, \tau_2 R_2 \}$, we have $K_\omega^\delta \subset K_\omega^{\tau_1 R_1} \cap K_\omega^{\tau_2 R_2} \subset \Omega$. By (\ref{sl2}), we have
$$
s^-(r, \theta; \tau_1, R_1) - s^-(r, \theta; \tau_2, R_2) = 0, \quad (r,\theta) \in K_\omega^\delta.
$$
Recall that $s^-(r, \theta; \tau_i, R_i) \in C^\infty(\Omega \setminus K_\omega^\delta)$, then it follows
$$
s^-(r, \theta; \tau_1, R_1) - s^-(r, \theta; \tau_2, R_2) \in C^\infty(\Omega).
$$
Since $\zeta(r, \theta; \tau_i, R_i)\in H_0^1(\Omega)$,  we have
\begin{align*}
\tilde{\xi} := & \xi(r, \theta; \tau_1, R_1) - \xi(r, \theta; \tau_2, R_2) \\
= &\zeta(r, \theta; \tau_1, R_1) - \zeta(r, \theta; \tau_2, R_2) + \left(s^-(r, \theta; \tau_1, R_1) - s^-(r, \theta; \tau_2, R_2) \right) \in H_0^1(\Omega).
\end{align*}
Meanwhile, from (\ref{seqn}), we have
$$
\Delta \tilde{\xi} = \Delta \xi(r, \theta; \tau_1, R_1) - \Delta \xi(r, \theta; \tau_2, R_2)=0 \text{ in } \Omega, \qquad \tilde{\xi} =0 \text{ on } \partial \Omega.
$$
By the Lax-Milgram Theorem, we have $\tilde{\xi} = 0$, and thus (\ref{paraindep}) holds.
\end{proof}

From now on we shall write $\xi(r,\theta)$ instead of $\xi(r, \theta; \tau, R)$, since it is independent of $\tau$ and $R$. We also notice   that $\xi(r, \theta)\not\equiv0$, because otherwise we have $ s^- =- \zeta \in H^1_0(\Omega)$, which contradicts  the fact that $s^-\notin H^1_0(\Omega)$.

\begin{remark}
In Lemma \ref{slem}, we have obtained $\xi \in L^2(\Omega)$ through (\ref{ssol}) instead of solving (\ref{seqn}) directly, which is due to the fact that the solution to (\ref{seqn}) is not unique in $L^2$ although it is uniquely defined in $H^1$. For example, the function $\xi\not\equiv 0$  as defined in (\ref{ssol}) and $\xi =0$  are both solutions to (\ref{seqn}) in $L^2(\Omega)$.
\end{remark}

Now we are ready to describe the subspace $\mathcal M^\perp$.
For the dimension of $\mathcal M^\perp$, we have the following result \cite{Grisvard1992}.
\begin{lem}\label{lem92}
The dimension of $\mathcal{M}^\perp$ is equal to the cardinality of the set $\{\lambda_k: 0<\lambda_{k}<1\}$ for $k\geq 1$, namely
\ben
\dim(\mathcal{M}^\perp) = \text{card}\{\lambda_{k} : 0<\lambda_{k}<1\},
\een
with $\lambda_{k}^2$ being the eigenvalues to the following one dimensional problem
\ben
-\partial_{\theta \theta}\phi_{k} = \lambda_{k}^2 \phi_{k} \quad \text{ in } (0,\omega), \qquad \phi(0)=\phi(\omega) =0.
\een
\end{lem}

For $k\geq 1$, it is clear that when $\lambda_k>0$,
\be\label{eigen}
\lambda_{k} = \frac{k\pi}{\omega}, \quad \phi_{k}= \sqrt{\frac{2}{\omega}} \sin \left(\frac{k\pi}{\omega} \theta\right).
\ee
Hence, for the domain $\Omega$ with one reentrant corner, $\mathcal M^\perp$ satisfies the following theorem.
\begin{theorem}\label{Mperp}
The dimension of $\mathcal{M}^\perp$ is $\dim(\mathcal{M}^\perp)=1$ and $\mathcal{M}^\perp = \text{span}\{\xi(r,\theta)\}$,
where $\xi(r,\theta)$ is the $L^2$ function defined in (\ref{ssol}).
\end{theorem}
\begin{proof}
Based on  Lemma \ref{lem92}, to find the dimension of $\mathcal{M}^\perp$, we only need to find the number of the integer(s) $k\geq 1$ such that $0<\lambda_{k}<1$.
According to (\ref{eigen}), $\lambda_{k}= \frac{k\pi}{\omega}<1$ implies
$
k < \frac{\omega}{\pi}.
$
Since $\frac{\omega}{\pi} \in (1, 2)$, thus $k=1$, namely $\dim(\mathcal{M}^\perp)=1$.

For $\xi$ defined in (\ref{ssol}), using Green's Theorem we have
\be\label{orthogonal}
(\Delta v, \xi) = (v, \Delta \xi) = 0, \quad \forall v \in H^2(\Omega)\cap H^1_0(\Omega),
\ee
where we have used (\ref{seqn}) for the second equality.
Note that $\Delta v \in \mathcal{M}$. Therefore,  (\ref{orthogonal}) implies $\xi \in \mathcal{M}^\perp$. Since $\xi \not = 0$ and $\dim(\mathcal{M}^\perp)=1$, we obtain the conclusion of the theorem.
\end{proof}

By Theorem \ref{Mperp}, for any $w \in L^2(\Omega)$ it can be uniquely expressed as
$$
w=w_\mathcal{M}+ c \xi,
$$
where $w_\mathcal{M} = w - c \xi\in \mathcal{M}$ and the coefficient
\be\label{scoef}
c=\frac{(w,\xi)-(w_\mathcal{M},\xi)}{\|\xi\|^2}=\frac{(w,\xi)}{\|\xi\|^2}.
\ee
\subsection{The modified mixed formulation}
Based on the discussion above, we propose a modified mixed formulation for (\ref{eqnbi}),
\begin{eqnarray}\label{eqnnew}
\left\{\begin{array}{ll}
-\Delta w=f \quad {\rm{in}} \ \Omega,\\
\hspace{0.6cm}w=0 \quad {\rm{on}} \ \pa\Omega;
\end{array}\right.
\qquad \qquad {\rm{and}}\qquad \qquad\left\{\begin{array}{ll}
-\Delta \tilde{u}=w - c \xi \quad {\rm{in}} \ \Omega,\\
\hspace{0.6cm}\tilde{u}=0 \quad {\rm{on}} \ \pa\Omega,
\end{array}\right.
\end{eqnarray}
where $\xi$ is given in (\ref{ssol}) and the coefficient $c$ is shown in (\ref{scoef}).
The corresponding modified mixed weak formulation for (\ref{eqnnew}) is to find $\tilde{u}, w\in H_0^1(\Omega)$ such that
\begin{subequations}\label{weaksysnew}
\begin{align}
A(w,\phi) = & (f, \phi), \\
A(\tilde{u}, \psi) = & (w- c \xi, \psi),
\end{align}
\end{subequations}
for any $\phi, \psi \in H_0^1(\Omega)$.

Then we have the following result for the modified mixed formulation.
\begin{theorem}\label{solequthm}
Let $\tilde{u}$ be the solution of the modified mixed weak formulation (\ref{weaksysnew}) and $u$ be the solution of the weak formulation (\ref{eqn.firstbi}). Then  $u=\tilde{u}$ in $H^2(\Omega)\cap H^1_0(\Omega)$.
\end{theorem}
\begin{proof}
Since $w \in H_0^1(\Omega) \subset L^2(\Omega)$, we have $w-c \xi \in \mathcal{M}$, which implies $\tilde{u} \in H^2{(\Omega)}\cap H^1_0(\Omega)$.
Thus (\ref{weaksysnew}b) becomes
\be\label{weakmb}
-(\Delta \tilde{u}, \psi) = \left(w-c \xi, \psi \right), \quad \forall \psi \in H_0^1(\Omega).
\ee
Note $\Delta \tilde{u} \in L^2(\Omega)$. Then following the density argument, (\ref{weakmb}) leads to
\begin{align*}
-(\Delta \tilde{u}, \psi) = \left(w- c \xi, \psi \right), \quad \forall \psi \in L^2(\Omega).
\end{align*}
Thus, for any $ \phi \in H^2(\Omega)\cap H^1_0(\Omega)$, we have $\Delta \phi \in L^2(\Omega)$ and therefore
\be\label{weakwtow}
(\Delta \tilde{u}, \Delta \phi) = \left(w-c \xi, -\Delta \phi \right).
\ee
Recall from (\ref{orthogonal}) that $\left(\xi, \Delta \phi \right)=0$. Then the right hand side of (\ref{weakwtow}) becomes
$$
\left(w-c \xi, -\Delta \phi \right)=A(w,\phi)+ \left(c \xi, \Delta \phi \right)=A(w,\phi)=(f,\phi),
$$
where the last equation is based on (\ref{weaksysnew}a). Hence, we have obtained that $\tilde u\in H^2(\Omega)\cap H^1_0(\Omega)$ satisfies
$$
(\Delta \tilde{u}, \Delta \phi) =(f,\phi), \quad \forall \phi \in H^2(\Omega)\cap H^1_0(\Omega).
$$
This is the same equation as  (\ref{eqn.firstbi}) that defines $u$. Consequently, $\tilde u=u\in H^2(\Omega)\cap H^1_0(\Omega)$ and we have completed the proof.
\end{proof}

Therefore, by Theorem \ref{solequthm}, the solution $u$ of the biharmonic problem (\ref{eqnbi}) satisfies
\begin{eqnarray}\label{eqnnew+}
\left\{\begin{array}{ll}
-\Delta w=f \quad {\rm{in}} \ \Omega,\\
\hspace{0.6cm}w=0 \quad {\rm{on}} \ \pa\Omega;
\end{array}\right.
\qquad \qquad {\rm{and}}\qquad \qquad\left\{\begin{array}{ll}
-\Delta {u}=w - c \xi \quad {\rm{in}} \ \Omega,\\
\hspace{0.6cm}{u}=0 \quad {\rm{on}} \ \pa\Omega.
\end{array}\right.
\end{eqnarray}
The corresponding weak formulation is to find $u, w\in H_0^1(\Omega)$ such that for any $\phi, \psi \in H_0^1(\Omega)$,
\begin{subequations}\label{weaknew}
	\begin{align}
		A(w,\phi) = & (f, \phi), \\
		A(u, \psi) = & (w- c \xi, \psi),
	\end{align}
\end{subequations}
where $c$ is given in (\ref{scoef}).

In addition,  we have the following regularity result.
\begin{lemma}\label{lemsreg}
Given $f \in L^2(\Omega)$, for $w,u$ in (\ref{weaknew}), it follows
\begin{subequations}\label{weakreg0}
\begin{align}
\|w\|_{H^1(\Omega)} \leq & C \|f\|,\\
\|u\|_{H^2(\Omega)} \leq & C \|f\|.
\end{align}
\end{subequations}
\end{lemma}
\begin{proof}
The estimate (\ref{weakreg0}a) is a direct consequence of the fact that the Laplace operator is an isomorphism between $H^1_0(\Omega)$ and $H^{-1}(\Omega)$ and $\|f\|_{H^{-1}(\Omega)}\leq C\|f\|$. For $u \in H^2(\Omega) \cap H_0^1(\Omega)$, from \cite[Theorem 2.2.3]{Grisvard1992}, there exist a constant $C$ such that
$$
\|u\|_{H^2(\Omega)} \leq C\|\Delta u\|.
$$
Note that by (\ref{scoef}),  $|c|\leq \|w\|/\|\xi\|$. Therefore,
$$
\|\Delta u\| = \|w- c \xi\| \leq C(\|w\|+|c|\|\xi\|) \leq 2C\|w\|\leq 2C\|w\|_{H^1(\Omega)}.
$$
Then the estimate (\ref{weakreg0}b) is proved by (\ref{weakreg0}a) and the estimates above.
\end{proof}

\section{The finite element method}\label{sec-3}

In this section, we propose a linear $C^0$ finite element method solving the biharmonic problem (\ref{eqnbi}). Then we derive the finite element error analysis for the solution $u$ to show that our method shall achieve the optimal convergence rate especially when the domain is non-convex.

\subsection{The finite element algorithm}\label{fem}
 Let $\maT_n$ be a triangulation of $\Omega$ with shape-regular triangles and let $S_n\subset H^1_0(\Omega)$ be the $C^0$ Lagrange finite element space associated with $\maT_n$,
\be\label{eqn.fems}
S_n(\maT):=\{v\in C^0(\Omega) \cap H_0^1(\Omega): \ v|_T\in P_1, \ \forall T \in \maT_n\},
\ee
where $P_1$ is the space of polynomials of degree no more than $1$. Then we proceed to propose the  finite element algorithm.




\begin{algorithm}\label{femalg}
We define the finite element solution of  the biharmonic problem (\ref{eqnbi}) by utilizing the decoupling in (\ref{weaknew}) as follows.
\begin{itemize}
\item{Step 1.}  Find the finite element solution $w_n\in S_n$ of the Poisson equation
\be\label{femw}
A(w_n, \phi)=(f, \phi), \qquad \forall \phi\in S_n.
\ee
\item{Step 2.} With $s^-$ defined in (\ref{sl2}), we compute the finite element solution $\zeta_n\in S_n$ of the Poisson equation
\be\label{femb}
A(\zeta_n, \phi)=(\Delta s^-, \phi), \qquad \forall \phi\in S_n,
\ee
and set $\xi_n=\zeta_n+ s^-$.
\item{Step 3.} Find the coefficient $c_n\in \mathbb R$, such that
 \begin{eqnarray*}
\int_{\Omega} (w_n-c_n\xi_n)\xi_ndx=0,
\end{eqnarray*}
or equivalently, we compute the coefficient
\be\label{scoefh}
c_n=\frac{(w_n,\xi_n)}{\|\xi_n\|^2}.
\ee
\item{Step 4.} Find the finite element solution $u_n\in S_n$ of the Poisson equation
\be\label{femu}
A(u_n, \psi)=(w_n-c_n\xi_n, \psi), \qquad \forall \psi\in S_n.
\ee
\end{itemize}
\end{algorithm}

\begin{remark} The function $s^-$ in (\ref{sl2}) exists only in the presence of a reentrant corner. When the domain is convex,  we set $s^-=0$ and  Algorithm \ref{femalg} reduces to the usual mixed finite element algorithm for equation (\ref{eqnbi}). According to (\ref{femb}), $\zeta_n\in S_n$,  while $\xi_n \in L^2(\Omega)$ but $\xi_n \not\in S_n$. In practice, $\xi_n$ can  be approximated by using high-order quadrature rules. In addition, the finite element approximations in Algorithm \ref{femalg} are well defined based on the Lax-Milgram Theorem.
\end{remark}

\subsection{Optimal error estimates on quasi-uniform meshes} 

Suppose that the mesh $\maT_n$ consists of quasi-uniform triangles with size $h$. Recall the interpolation error estimate on  $\maT_n$ \cite{Ciarlet78} for any $v \in H^l(\Omega)$, $l>1$,
\be\label{interr}
\| v - v_I \|_{H^m(\Omega)} \leq Ch^{l-m}\|v\|_{H^l(\Omega)},
\ee
where $m= 0, 1$ and $v_I\in S_n$ represents the nodal interpolation of $v$.
For the Poisson equations (\ref{l2part}) and (\ref{eqnnew+}) in the polygonal domain with a reentrant corner, given $f\in L^2(\Omega)$, it is well known that $w, \zeta \in H^\alpha(\Omega)$ with $\alpha<1+\frac{\pi}{\omega}$ (see for example \cite{Grisvard1985, Grisvard1992}).
Recall  the finite element approximations $w_n$ and $\zeta_n$  in (\ref{femw}) and (\ref{femb}), respectively.
Due to the lack of regularity, the standard error estimate  \cite{Ciarlet78} yields
\be\label{wh1err}
\|w-w_n\|_{H^1(\Omega)} \leq Ch^{\alpha-1}\|w\|_{H^\alpha(\Omega)}, \quad \|\zeta-\zeta_n\|_{H^1(\Omega)} \leq Ch^{\alpha-1}\|\zeta\|_{H^\alpha(\Omega)}.
\ee
Note that $\xi-\xi_h=\zeta-\zeta_h \in H^\alpha(\Omega)$, and thus
$$
\|\xi-\xi_n\|_{H^1(\Omega)} \leq Ch^{\alpha-1}\|\zeta\|_{H^\alpha(\Omega)}.
$$

In addition, we have the following $L^2$ error analysis.

\begin{lemma}\label{lemwerrl2}
Given $w_h$ and $\xi_h$ in Algorithm \ref{femalg}, we have
\be\label{wl2err}
\|w-w_n\| \leq Ch^{2\alpha-2}\|w\|_{H^\alpha(\Omega)}, \quad \|\xi-\xi_n\| \leq Ch^{2\alpha-2}\|\zeta\|_{H^\alpha(\Omega)}.
\ee
\end{lemma}
\begin{proof}
We only prove the error estimate for $w-w_h$, and the  estimate for $\xi-\xi_n$ can be obtained similarly. 
Consider the Poisson problem \be\label{dualVL2}
-\Delta v =g \text{ in } \Omega, \quad v =0 \text{ on } \partial \Omega,
\ee
where $g \in L^2(\Omega)$.
By the Aubin-Nitsche Lemma in \cite[Theorem 3.2.4]{Ciarlet78}, we have
\be\label{wel2}
\bal
\|w-w_h\| \leq  C \|w-w_h\|_{H^1(\Omega)} \sup_{g \in L^2(\Omega)} \left( \frac{\inf_{\phi \in S_n}\|v-\phi\|_{H^1(\Omega)}}{\|g\|}  \right).
\eal
\ee
Since
$
\|v\|_{H^\alpha(\Omega)} \leq C\|g\|,
$
we have
\be\label{projl2}
\bal
\inf_{\phi \in S_n}\|v-\phi\|_{H^1(\Omega)} \leq \|v-v_I\|_{H^1(\Omega)} \leq Ch^{\alpha-1}\|v\|_{H^\alpha(\Omega)}\leq Ch^{\alpha-1}\|g\|.
\eal
\ee
Combining  the estimates in (\ref{projl2}),  (\ref{wel2}), and (\ref{wh1err}), we have completed the proof.
\end{proof}

Next we carry out the error estimate for the finite element approximation $u_n$ in (\ref{femu}).
\begin{theorem}\label{thmuerr}
Let $u_n \in S_n$ be the finite element approximation to (\ref{femu}), and $u$ be the solution to the biharmonic problem (\ref{eqn.firstbi}). Then it follows
\bes
\|u-u_n\|_{H^1(\Omega)} \leq Ch.
\ees
\end{theorem}
\begin{proof}
For any $\phi\in S_n$, based on (\ref{weaknew}b) and (\ref{femu}), we have
\ben
\bal
A(u, \phi)=& (w, \phi)-c(\xi, \phi),\\
A(u_n, \phi)=&(w_n, \phi)-c_n(\xi_n, \phi).
\eal
\een
Taking difference of the two equations above, we have
\be\label{soldiff}
\bal
A(u-u_n, \phi)=&(w-w_n, \phi)+c_n(\xi_n, \phi)-c(\xi, \phi)\\
=&(w-w_n, \phi)+c_n(\xi_n-\xi, \phi)+(c_n-c)(\xi, \phi).
\eal
\ee
Let $u_I\in S_n$ be the nodal interpolation of $u$. Set $\epsilon = u_I - u, \ e = u_I - u_n$ and take $\phi=e$ in (\ref{soldiff}). We have
\ben
\bal
A(e, e)=A(\epsilon,e) + (w-w_n, e)+c_n(\xi_n-\xi, e)+(c_n-c)(\xi, e).
\eal
\een
Thus, we have
$$
\|e\|_{H^1(\Omega)}^2 \leq C \left( \|\epsilon\|_{H^1(\Omega)} +  \|w-w_n\|_{H^{-1}(\Omega)} + |c_n|\|\xi_n-\xi\|_{H^{-1}(\Omega)} + |c-c_n|\|\xi\|_{H^{-1}(\Omega)}  \right) \|e\|_{H^1(\Omega)}.
$$
Using the triangle inequality and the inequality above, we have
\begin{eqnarray}\label{errbdd}
\|u-u_n\|_{H^1(\Omega)} &\leq& \|e\|_{H^1(\Omega)} +  \|\epsilon\|_{H^1(\Omega)}\nonumber\\
&\leq & C \left( \|\epsilon\|_{H^1(\Omega)} +  \|w-w_n\|_{H^{-1}(\Omega)} + |c_n|\|\xi_n-\xi\|_{H^{-1}(\Omega)} + |c-c_n|\|\xi\|_{H^{-1}(\Omega)}  \right)\nonumber\\
&\leq & C \left( \|\epsilon\|_{H^1(\Omega)} +  \|w-w_n\| + |c_n|\|\xi_n-\xi\| + |c-c_n|\|\xi\|  \right). \label{uh1err}
\end{eqnarray}
The last inequality is based on the fact that the $H^{-1}$ norm of an $L^2$ function  is bounded by its $L^2$ norm. We shall estimate every term in (\ref{uh1err}). Recall the solution $u\in H^2(\Omega)$. By the interpolation error estimate (\ref{interr}),
\be\label{uprojerr}
\|\epsilon\|_{H^1(\Omega)} = \|u-u_I\|_{H^1(\Omega)} \leq Ch\|u\|_{H^2(\Omega)}.
\ee
\tb{
Recall the angle of the reentrant corner $\pi<\omega<2\pi$.  Thus, choosing $\alpha=3/2<1+\frac{\pi}{\omega}$ in (\ref{wl2err}), we have
\be\label{eqn.11}
\|w-w_n\|\leq Ch, \quad \|\xi-\xi_n\|\leq Ch.
\ee
Recall that $\xi\not \equiv 0$ depends only on the domain $\Omega$, thus it follows
\be\label{xi+}
\|\xi\| > 0.
\ee
Moreover, when $h \leq h_0:=\min \left\{1, \sqrt[2\alpha-2]{\frac{\|\xi\|}{2C\|\xi\|_{H^\alpha(\Omega)}}}\right\}$, it follows from (\ref{wl2err}) that
\be\label{xibdd}
\frac{1}{2}\|\xi\| \leq \|\xi_n\|\leq \frac{3}{2}\|\xi\|,
\ee
so we have
\be\label{xinvbdd}
\|\xi_n\|^{-1} \leq C_0,
\ee
where the constant $C_0$ depends only on $\xi$ or $\Omega$. By setting $\phi = w_n$ in (\ref{femw}) and applying the Poincar\'{e} inequality, we obtain
\be\label{wnbdd}
\|w_n\|\leq C_1\|w_n\|_{H^1(\Omega)} \leq C_2\|f\|.
\ee
For the coefficient $|c_n|$ in (\ref{errbdd}), by (\ref{scoefh}), (\ref{xinvbdd}) and (\ref{wnbdd}), we have
\be\label{cbnd}
|c_n| \leq \frac{\|w_n\|}{\|\xi_n\|} \leq C\|f\|,
\ee
where $C$ is a constant depending on $\Omega$.
}
Subtracting (\ref{scoefh}) from (\ref{scoef}), we  obtain
$$
c-c_n = \frac{(w-w_n,\xi)}{\|\xi\|^2}+\frac{(\xi-\xi_n, w_n)}{\|\xi_n\|^2} + \frac{\|\xi_n\|^2-\|\xi\|^2}{\|\xi\|^2\|\xi_n\|^2}(w_n, \xi).
$$
\tb{By (\ref{xi+}), (\ref{xibdd}), (\ref{wnbdd}) and (\ref{eqn.11}), it follows}
\be\label{cdiffbnd}
\bal
|c-c_n| \leq \frac{1}{\|\xi\|}\|w-w_n\|+\frac{\|w_n\|}{\|\xi_n\|^2}\|\xi-\xi_n\|+  \frac{(\|\xi_n\|+\|\xi\|)\|w_n\|}{\|\xi\|\|\xi_n\|^2}\|\xi-\xi_n\| \leq Ch.
\eal
\ee
Then the proof is completed by plugging (\ref{uprojerr}), (\ref{cbnd}), (\ref{cdiffbnd}), and  (\ref{eqn.11}) 
into (\ref{uh1err}).
\end{proof}

\begin{remark}\label{rk35} The error estimate in Theorem \ref{thmuerr} shows that the proposed finite element algorithm (Algorithm \ref{femalg}) produces numerical solutions that converge to the solution of the biharmonic problem (\ref{eqnbi}) when the domain $\Omega$ is non-convex. In the case that $\Omega$ is convex, Algorithm \ref{femalg} reduces to the usual mixed finite element algorithm for equation (\ref{eqnbi}) that has proven to be effective \cite{MR2479119}. Therefore, Algorithm \ref{femalg} approximates the target equation in both convex and non-convex domains. On a quasi-uniform mesh, the convergence is first-order (optimal) for $u$ in the $H^1$ norm (Theorem \ref{thmuerr}) and sub-optimal  (\ref{wh1err}) for the auxiliary function $w$ in the $H^1$ norm. In Algorithm \ref{femalg}, one shall solve three Poisson problems. Given the availability of fast Poisson solvers, Algorithm \ref{femalg} is an relatively easy and cost effective alternative to existing algorithms solving (\ref{eqnbi}).
\end{remark}

\section{Optimal error estimates on graded meshes}\label{sec-4}
The numerical approximations from Algorithm \ref{femalg} are optimal for $u$ but only sub-optimal for $w$. It is largely due to the lack of regularity for  $w$. Recall that for $f\in L^2(\Omega)$, $w$ is merely in $H^{\alpha}(\Omega)$ for $0<\alpha<1+\frac{\pi}{\omega}$. In this section, we study the system (\ref{weaknew}) in a class of weighted Sobolev spaces and in turn propose graded triangulations that lead to numerical solutions converging in the optimal rate to both $u$ and $w$.

\subsection{Regularity in weighted Sobolev spaces}

We now  introduce  the Kondratiev-type weighted spaces for the analysis of the system (\ref{weaknew}).
\begin{definition} \label{wss} (Weighted Sobolev spaces) Recall that $Q$ is the vertex at the reentrant corner.
Let $r(x)$ be the distance from $x$ to $Q$.
For $a\in\mathbb R$, $m\geq 0$, and $G\subset \Omega$,  we define the weighted Sobolev space
$$
\maK_{a}^m(G) := \{v,\ r^{|\alpha|-a}\partial^\alpha v\in L^2(G),\ \forall |\alpha|\leq m \},
$$
where the multi-index $\alpha=(\alpha_1,\alpha_2)\in\mathbb Z^2_{\geq 0}$, $|\alpha|=\alpha_1+\alpha_2$, and $\partial^\alpha=\partial_{x}^{\alpha_1}\partial_{y}^{\alpha_2}$.
The $\maK_{a}^m(G)$ norm for $v$  is defined by
$$
\|v\|_{\maK_{a}^m(G)}=\big(\sum_{|\alpha|\leq m}\int_{G} |r^{|\alpha|-a}\partial^\alpha v|^2dx\big)^{\frac{1}{2}}.
$$
\end{definition}
\begin{remark}
According to Definition \ref{wss}, in the region that is away from the reentrant corner, the weighted space $\maK^m_a$ is equivalent to the Sobolev space $H^m$. In the neighborhood of $Q$, the space $\maK^m_a$ is the same Kondratiev space  \cite{Dauge88,Grisvard1985,Kondratiev67}. Recall the first equation in (\ref{eqnnew+}) that defines $w$. In the Dirichlet Poisson problem, the reentrant corner can give rise to singularities in $w$, such that $w\notin H^2(\Omega)$. It is the reason that the finite element approximation to $w$ on a quasi-uniform mesh is not optimal. The singularity in $w$ is however local and concentrates in the neighborhood of $Q$. Involving a proper weight function, the space $\maK^m_a$ may allow more singular functions and is an important tool for analyzing corner singularities.
\end{remark}


In the weighted Sobolev space, we have the following regularity result for the system (\ref{weaknew}).
\begin{lemma}\label{lemw}
Assume $a<\frac{\pi}{\omega}$ and $f \in L^2(\Omega)$. Recall $\zeta$ in (\ref{l2part}). Then it follows
\ben
\|\zeta\|_{\maK_{a+1}^2(\Omega)} \leq C \|\Delta s^-\|.
\een	
In addition, recall $w$ in (\ref{eqnnew+}). Then we have
\ben
\|w\|_{\maK_{a+1}^2(\Omega)} \leq  C \|f\|.
\een
\end{lemma}
\begin{proof}
Since $\Delta s^-, f\in L^2(\Omega)\subset \maK^{0}_{a-1}(\Omega)$,  the desired estimates follow by applying Theorem 3.3 of \cite{LN09} to equations (\ref{l2part}) and (\ref{eqnnew+}).
\end{proof}

\subsection{Graded meshes}

We now present the construction of graded meshes to improve the convergence rate of the numerical approximation from Algorithm \ref{femalg}.

\begin{algorithm}\label{graded} (Graded refinements) Let $\maT$ be a triangulation of $\Omega$ with shape-regular triangles.
 Recall that $Q$ is the vertex of $\Omega$ at the reentrant corner. It is clear that $Q$ is also a vertex in the triangulation $\mathcal T$.
Let ${pq}$ be an edge in the triangulation $\mathcal T$ with $p$ and $q$ as the endpoints. Then, in a graded refinement, a new node $r$ on $pq$ is produced according to the following conditions:
\begin{itemize}
\item[1.] (Neither $p$ or $q$ coincides with $Q$.) We choose $r$  as the midpoint ($|pr|=|qr|$).
\item[2.] ($p$ coincides with $Q$.) We choose $r$  such that $|pr|=\kappa|pq|$, where $\kappa\in (0, 0.5)$ is a parameter that will be specified later. See Figure \ref{fig.2} for example.
\end{itemize}
Then, the graded refinement, denoted by $\kappa(\mathcal T)$, proceeds as follows.
For each triangle $T\in \mathcal T$, a new node is generated on each edge of $T$ as described above. Then, $T$ is decomposed into four small triangles by connecting these new nodes (Figure \ref{fig.333}). Given an initial mesh $\mathcal T_0$ satisfying the condition above, the associated family of graded meshes $\{\mathcal T_n,\ n\geq0\}$ is defined recursively $\mathcal T_{n+1}=\kappa(\mathcal T_{n})$.
\end{algorithm}

\begin{figure}[h]
\includegraphics[scale=0.34]{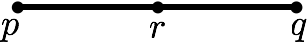}\hspace{3cm}\includegraphics[scale=0.34]{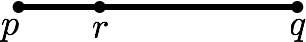}
\caption{The new node  on an edge $pq$ (left -- right): $p\neq Q$ and $q\neq Q$ (midpoint); $p=Q$  ($|pr|=\kappa|pq|$,  $\kappa<0.5$).}\label{fig.2} \end{figure}

\begin{figure}[h]
 \includegraphics[scale=0.34]{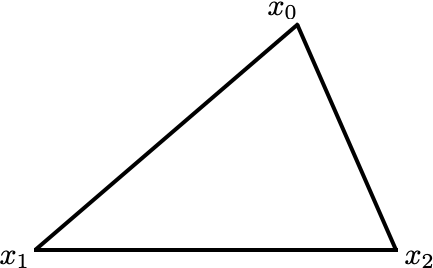}\hspace{0.5cm} \includegraphics[scale=0.34]{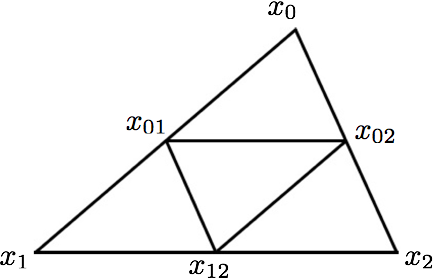}\\\hspace{0cm}\includegraphics[scale=0.34]{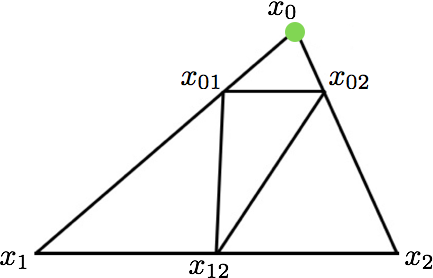}\hspace{.5cm}\includegraphics[scale=0.34]{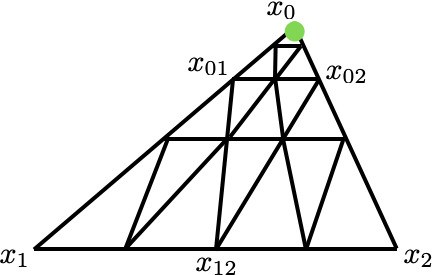}
\caption{Refinement of a triangle $\triangle x_0x_1x_2$. First row: (left -- right): the initial triangle and the midpoint refinement; second row: two consecutive graded refinements toward $x_0=Q$, ($\kappa<0.5$).}\label{fig.333}\end{figure}


Given a grading parameter $\kappa$, Algorithm \ref{graded} produces smaller elements near $Q$ for better approximation of singular solution. It is an explicit construction of graded meshes based on recursive refinements. See also \cite{NistorBZ052, LMN10, LN09} and references therein for more discussions on the graded mesh. Note that after $n$ refinements, the number of triangles in the mesh $\maT_n$ is $O(4^n)$.

\subsection{Optimal error estimates on graded meshes}
In the rest of this section, we shall show that with a proper selection of the grading parameter $\kappa$,  the proposed numerical solutions $u_n$ and $w_n$ converge to the solutions $u$ and $w$ of \eqref{eqnnew+} in the optimal rate on graded meshes. Recall the finite element space $S_n$ in (\ref{eqn.fems})  associated with the graded mesh $\maT_n$.

We first recall the following interpolation error estimates \cite{LN09} for functions in the weighted space.
\begin{lemma}\label{interpolation}
Let $0<a<\frac{\pi}{\omega}$ and choose the grading parameter $\kappa= 2^{-1/a}$. Define $h:=2^{-n}$.  Then  for any $v \in \maK_{a+1}^2(\Omega)$, it follows
\ben
\|v-v_I\|_{H^1(\Omega)} \leq Ch \|v\|_{\maK_{a+1}^2(\Omega)},
\een
where $v_I$ be the nodal interpolation of $v$ associated with $\maT_n$.
\end{lemma}

\tb{Recall that $w, \zeta \in H^\alpha(\Omega)$ with $\alpha<1+\frac{\pi}{\omega}$, so the error estimates in (\ref{wh1err}) and (\ref{wl2err}) are still valid on graded meshes with $0<\kappa<\frac{1}{2}$, but it is suboptimal. For the optimal error estimates of $w-w_n$ and $\xi-\xi_n$ on graded meshes, we have the following result.}
\begin{lemma}\label{glemwerrl2} Let $0<a<\frac{\pi}{\omega}$ and choose $\kappa= 2^{-1/a}$. Then
for the approximations $w_n$ and $\xi_n$ defined in (\ref{femw}) and (\ref{femb}),
 it follows
\begin{subequations}\label{gwl2err}
\begin{align}
\|w-w_n\|_{H^1(\Omega)} \leq Ch\|w\|_{\maK_{a+1}^2(\Omega)},& \quad \|w-w_n\| \leq Ch^2\|w\|_{\maK_{a+1}^2(\Omega)},\\
\|\xi-\xi_n\|_{H^1(\Omega)} \leq Ch\|\zeta\|_{\maK_{a+1}^2(\Omega)},& \quad \|\xi-\xi_n\| \leq Ch^2\|\zeta\|_{\maK_{a+1}^2(\Omega)},
\end{align}
\end{subequations}
where  $h :=2^{-n}$.
\end{lemma}
\begin{proof}
We only prove the error estimates for $w-w_n$ and the estimates for $\xi-\xi_n$ will follow similarly.
Let $w_I\in S_n$ be the nodal interpolation of $w$.
Using the regularity estimate in Lemma \ref{lemw} and the interpolation error estimate in Lemma \ref{interpolation}, we derive the analysis for  $\|w-w_n\|_{H^1(\Omega)}$  as follows:
$$
\|w-w_n\|_{H^1(\Omega)}\leq C\|w-w_I\|_{H^1(\Omega)}\leq Ch\|w\|_{\maK_{a+1}^2(\Omega)}.
$$

Meanwhile, applying the Aubin-Nitsche Lemma to (\ref{dualVL2}) again, we have
\be\label{gwel2}
\bal
\|w-w_n\| \leq  C \|w-w_n\|_{H^1(\Omega)} \sup_{g \in L^2(\Omega)} \left( \frac{\inf_{\phi \in S_n}\|v-\phi\|_{H^1(\Omega)}}{\|g\|}  \right).
\eal
\ee
Based on the regularity estimate in Lemma \ref{lemw}, the function $v$ in (\ref{dualVL2}) satisfies
$$
\|v\|_{\maK_{a+1}^{2}(\Omega)} \leq C \|g\|.
$$
Together with Lemma \ref{interpolation}, we have
\be\label{gprojl2}
\bal
\inf_{\phi \in S_n}\|v-\phi\|_{H^1(\Omega)} \leq \|v-v_I\|_{H^1(\Omega)} \leq Ch\|v\|_{\maK_{a+1}^{2}(\Omega)}\leq Ch\|g\|,
\eal
\ee
where $v_I$ be the nodal interpolation of $v$ associated with $\maT_n$.
Plugging (\ref{gprojl2}) into (\ref{gwel2}) leads to the desired error estimate for $\|w-w_n\|$.
\end{proof}

We conclude this section by the $H^1$ error estimate for the solution $u$ of the biharmonic problem (\ref{eqnbi}) on graded meshes.

\begin{theorem}\label{gthmuerr} 
\tb{Let $0<\kappa<\frac{1}{2}$ for the mesh $\maT_n$.}
Let $u_n$ be the finite element approximation to $u$ that is defined in Algorithm \ref{femalg}. Then it follows
\bes
\|u-u_n\|_{H^1(\Omega)} \leq Ch,
\ees
where $h:=2^{-n}$.
\end{theorem}
\begin{proof}
Let $u_I$ be the nodal interpolation of $u$ associated with $\maT_n$. Similar to the analysis in Theorem \ref{thmuerr} on quasi-uniform meshes, we have
\be\label{guh1err}
\bal
\|u-u_n\|_{H^1(\Omega)} \leq   C \left( \|u- u_I\|_{H^1(\Omega)} +  \|w-w_n\| + \frac{\|w_n\|}{\|\xi_n\|}\|\xi_n-\xi\| + |c-c_n|\|\xi\|_{H^{-1}(\Omega)}  \right).
\eal
\ee
For $0<\kappa<\frac{1}{2}$, the following interpolation error still holds
\be\label{guprojerr}
\|u-u_I\|_{H^1(\Omega)} \leq Ch\|u\|_{H^2(\Omega)}.
\ee
Thus, the proof is completed by combining the estimates in (\ref{guh1err}), (\ref{guprojerr}), (\ref{cdiffbnd}), and the $L^2$ error estimates for $w-w_n$ and $\xi-\xi_n$ in (\ref{wl2err}).
\end{proof}

\begin{remark} According to Theorem \ref{thmuerr} and Theorem \ref{gthmuerr}, the numerical solution $u_n$ in Algorithm  \ref{femalg}  approximates the solution $u$ of the biharmonic problem in the optimal $H^1$ convergence rate on quasi-uniform meshes and also on graded meshes defined in Algorithm \ref{graded}. Meanwhile, a proper graded mesh can improve the effectiveness in approximating the auxiliary function $w$ in (\ref{eqnnew+}). In particular, selecting the grading parameter $\kappa$ as in Lemma \ref{glemwerrl2}, the proposed finite element solution $w_n$ converges to $w$ in both $H^1$ and $L^2$ norms with the optimal rate on graded meshes. Nonetheless, the numerical approximations $u_n$ and $w_n$ from Algorithm \ref{femalg} converge to $u$ and $w$ in both convex and non-convex domains. The graded mesh can improve the convergence rate  but does not make divergent numerical solutions convergent.
 We also point out that when high-order finite element methods are used  in Algorithm \ref{femalg}, new graded meshes are needed to recover the optimal $H^1$ convergence rate for both $u$ and $w$. We shall study these cases in future works.
\end{remark}

\section{Numerical illustrations}\label{sec-5}
In this section, we present numerical test results to validate our theoretical predictions for the proposed finite element method solving equation (\ref{eqnbi}). Since the solutions $u, w$ in (\ref{eqnnew+}) are unknown, we use the following numerical convergence rate
\begin{eqnarray}\label{rate}
{\mathcal R}=\log_2\frac{|v_j-v_{j-1}|_{H^1(\Omega)}}{|v_{j+1}-v_j|_{H^1(\Omega)}}
\end{eqnarray}
as an  indicator of the actual convergence rate.
Here $v_j$ denotes the finite element solution on the  mesh $\mathcal T_j$ obtained after $j$  refinements of the initial triangulation $\mathcal T_0$. It can be either $u_j$ or $w_j$ depending on the underlying Poisson problem. In particular, suppose the actual convergence rate is $\|v-v_j\|_{H^1(\Omega)}=O(h^\beta)$ for $\beta>0$. Then for the $P_1$ finite element method, the rate in (\ref{rate}) is a good approximation of the exponent $\beta$ as the level of refinements $j$ increases  \cite{LN18}.
%

We shall use the solution of the $C^0$ interior penalty method  in FEniCS \cite{fenics} as a reference solution. More specifically, we use the penalty method based on $P_2$ polynomials with the penalty parameter $\eta=24$.  The reference solution is computed on the mesh after seven mesh refinements of the given initial mesh and is denoted by $u_R$. Since the $C^0$ interior penalty method leads to numerical solutions converging to the solution $u$ regardless of the convexity of the domain, we can use $u_R$ as a good approximation of $u$. We point out that given the same triangulation, the implementation of the proposed finite element algorithm (Algorithm \ref{femalg}) can be much faster than that of the $C^0$ interior penalty method, due  to the availability of fast Poisson solvers.

\begin{example}\label{P1S} (A convex domain).
We consider the problem (\ref{eqnbi}) with $f=10$ in the square domain $\Omega=(0,2)^2$. Since all the vertices have angles less than $\pi$,  Algorithm \ref{femalg} coincides with the usual mixed finite element method based on the  formulations in (\ref{weaksys}).

We solve this problem using Algorithm \ref{femalg} on uniform meshes obtained by midpoint refinements with the initial mesh  given in Figure \ref{Mesh_Init0}(a). The finite element solution $u_7$ and the difference $|u_R-u_7|$ are shown in Figures \ref{Mesh_Init0}(c) and \ref{Mesh_Init0}(d), respectively. The convergence rates (\ref{rate}) for $u_j$ and $w_j$ on a sequence of uniform meshes are shown in Table \ref{TabConRateS}. We see that the solution of the mixed finite element method  converges to the solution  of the biharmonic equation (\ref{eqnbi}) and the optimal  convergence rate ($\mathcal R=1$)  is achieved for both the numerical solution $u_j$ and the auxiliary finite element solution $w_j$. This is consistent with our expectation (Remark \ref{rk26}) for the problem in a convex domain.

\begin{figure}[h]
\centering
\subfigure[]{\includegraphics[width=0.222\textwidth]{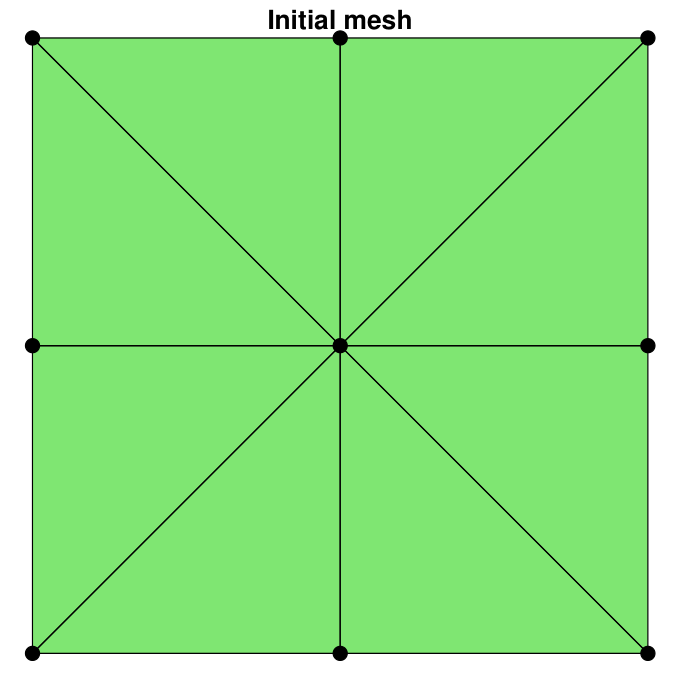}}
\subfigure[]{\includegraphics[width=0.222\textwidth]{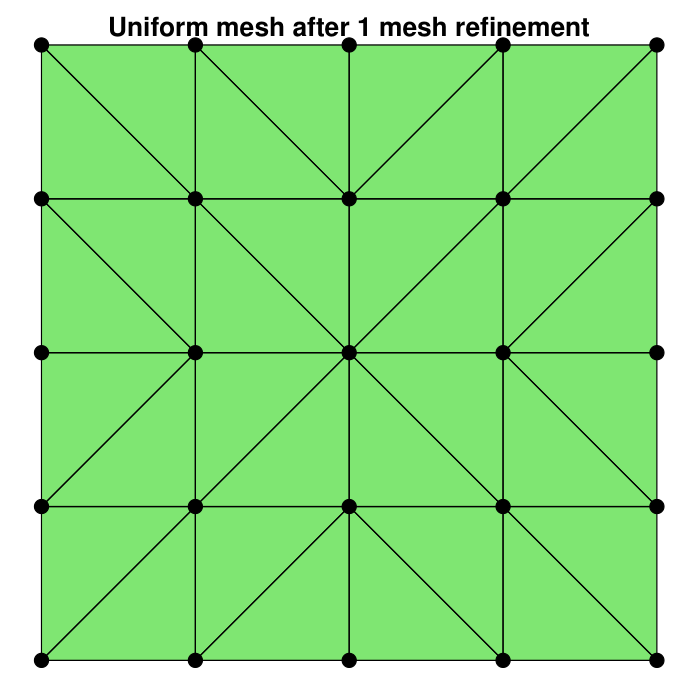}}
\subfigure[]{\includegraphics[width=0.252\textwidth]{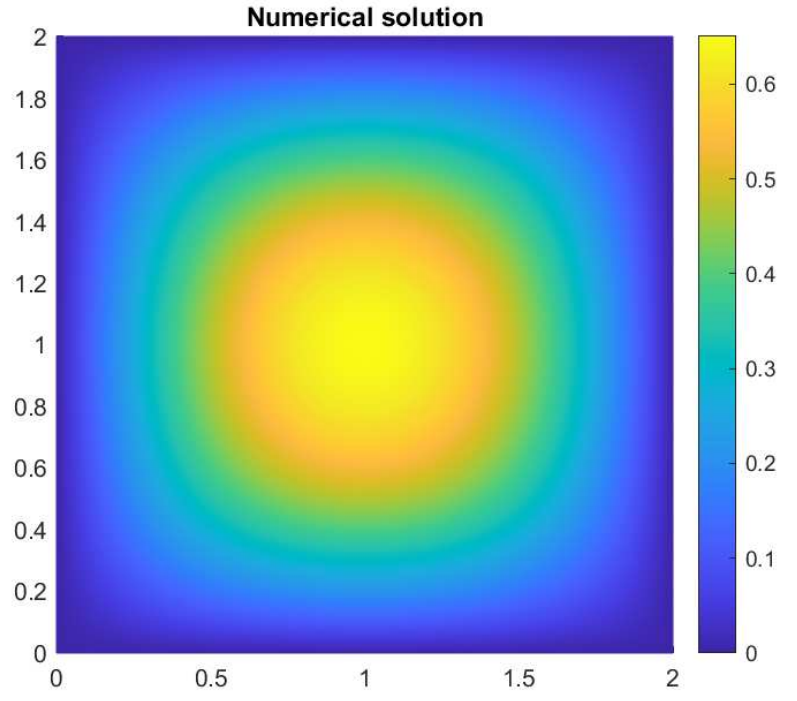}}
\subfigure[]{\includegraphics[width=0.255\textwidth]{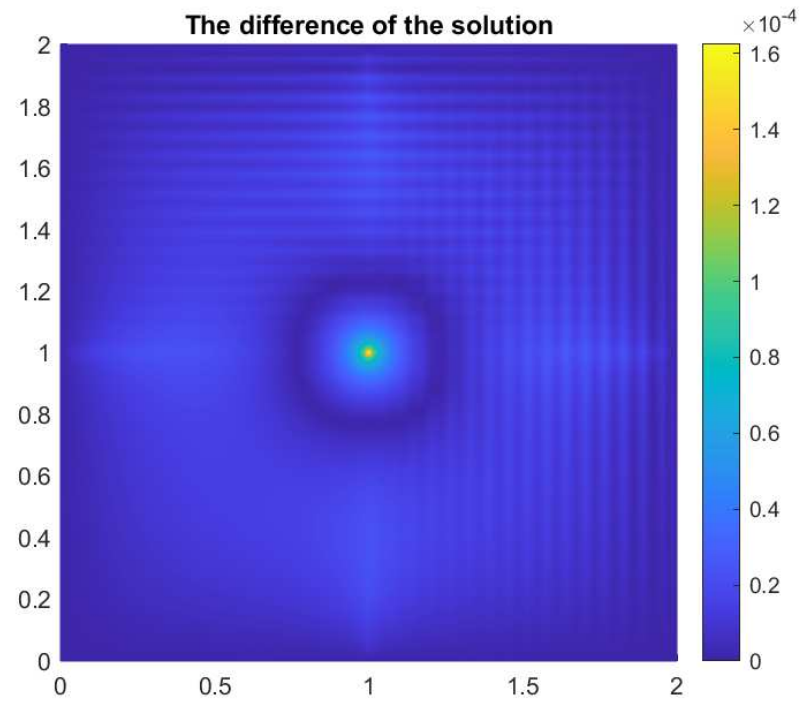}}
\caption{The square domain (Example \ref{P1S}): (a) the initial mesh; (b) the mesh after one refinement; (c): the solution $u_7$ from Algorithm \ref{femalg}; (d) $|u_R-u_7|$. }\label{Mesh_Init0}
\end{figure}

\begin{table}[!htbp]\tabcolsep0.03in
\caption{$H^1$ Convergence history  of the $P_1$ elements in Example \ref{P1S} on uniform meshes.}
\begin{tabular}[c]{|c|cccc|}
\hline
 & $j=3$ & $j=4$ & $j=5$ & $j=6$ \\
\hline
$\mathcal R$ for $u_j$ & 0.96 & 0.99 & 1.00 & 1.00 \\
\hline
$\mathcal R$ for $w_j$ & 0.96 & 0.99 & 1.00 & 1.00 \\
\hline
\end{tabular}\label{TabConRateS}
\end{table}

\end{example}

\begin{example}\label{P1h} (A non-convex domain).  In this example, we investigate the convergence of Algorithm \ref{femalg} by considering equation  (\ref{eqnbi}) with $f=1$ in an L-shaped domain $\Omega=\Omega_0 \setminus \Omega_1$ with $\Omega_0=(-2,2)^2$ and $\Omega_1=(0,2)\times (-2,0)$.
We use  the following cut-off function in the algorithm:
\begin{eqnarray*}
\eta(r; \tau, R)=
\left\{\begin{array}{ll}
0, & \text{if } r \geq R, \\
1, & \text{if } r \leq \tau R, \\
\frac{1}{2} - \frac{15}{16}\left( \frac{2r}{R(1-\tau)}-\frac{1+\tau}{1-\tau}\right) + \frac{5}{8}\left( \frac{2r}{R(1-\tau)}-\frac{1+\tau}{1-\tau}\right)^3 - \frac{3}{16}\left( \frac{2r}{R(1-\tau)}-\frac{1+\tau}{1-\tau}\right)^5 , & \text{otherwise,}
\end{array}\right.
\end{eqnarray*}
where $R=\frac{9}{5}, \tau = \frac{1}{8}$.


In the first test, we solve equation (\ref{eqnbi})  in the L-shaped domain using quasi-uniform meshes and compare the performances of  Algorithm \ref{femalg} and the  usual mixed finite element algorithm based on the  formulation (\ref{weaksys}). On $\maT_j$, we denote the numerical solutions  from Algorithm \ref{femalg} by $u^A_j$ and $w_j^A$, and denote the numerical solutions from the usual mixed finite element algorithm by $u^U_j$ and $w^U_j$.  The  initial mesh is shown in Figure \ref{Mesh_Init}(a).
In Table \ref{DiffErr}, we display the   errors ($u_R-u^U_j$ and $u_R-u^A_j$) in the $L^\infty$ norm between the finite element solutions  and the reference solution  $u_R$.
In addition, the differences $|u_R-u^U_7|$ and $|u_R-u^A_7|$ are also presented in Figures \ref{Mesh_Init}(d) and \ref{Mesh_Init}(f). From these results, we see that the solution $u_j^A$ from Algorithm \ref{femalg} converges to the actual solution, while the solution $u_j^U$ of the usual mixed finite element algorithm {does not converge to the solution of the biharmonic equation (\ref{eqnbi})} as the meshes are refined.  These observations are  closely aligned with our theoretical predictions in Remark \ref{rk35}. Namely, Algorithm \ref{femalg} gives rise to convergent numerical solutions in both convex and non-convex domains, while the usual mixed method is applicable  only for convex domains.

\begin{table}[!htbp]\tabcolsep0.03in
\caption{The $L^\infty$ error in the L-shaped domain  on quasi-uniform meshes.}
\begin{tabular}[c]{|c|c|c|c|c|}
\hline
\multirow{2}{*}{} & $j=3$ & $j=4$ & {$j=5$} & {$j=6$}  \\
\cline{2-5}
\hline
$\|u_R-u_j^U\|_{L^\infty(\Omega)}$  & 1.28014e-01 & 1.37318e-01 & 1.41309e-01 & 1.42525e-01 \\
\hline
$\|u_R-u_j^A\|_{L^\infty(\Omega)}$  &  1.58074e-02 & 7.84320e-03 & 3.20391e-03 & 1.20794e-03 \\
\cline{1-5}
\end{tabular}\label{DiffErr}
\end{table}

\begin{table}[!htbp]\tabcolsep0.03in
\caption{Numerical convergence rates $\mathcal R$ for $u^U_j$ and $w^U_j$ in the L-shaped domain.}
\begin{tabular}[c]{|c|ccccc|ccccc|}
\hline
\multirow{2}{*}{$\kappa\backslash j$}  & \multicolumn{5}{|c|}{$u_j^U$} & \multicolumn{5}{|c|}{$w_j^U$}  \\
\cline{2-11}
& $j=5$ & $j=6$ & $j=7$ & $j=8$ & $j=9$ & $j=5$ & $j=6$ & $j=7$ & $j=8$ & $j=9$\\
\hline
$\kappa=0.1$ & 0.95 & 0.98 & 0.99 & 1.00 & 1.00 & 0.95 & 0.98 & 0.99 & 1.00 & 1.00 \\
\hline
$\kappa=0.2$ & 0.96 & 0.99 & 1.00 & 1.00 & 1.00 & 0.96 & 0.99 & 0.99  & 1.00 & 1.00\\
\hline
$\kappa=0.3$ & 0.97 & 0.98 & 0.99 & 0.99 & 0.99 & 0.96 & 0.98 & 0.99 & 0.99  & 0.99 \\
\hline
$\kappa=0.4$ & 0.94 & 0.94 & 0.94 & 0.93 & 0.93 & 0.94 & 0.95 & 0.95  & 0.94 & 0.94 \\
\hline
$\kappa=0.5$ & 0.86 & 0.82 & 0.78 & 0.75 & 0.72 & 0.87 & 0.84 & 0.80  & 0.77 & 0.74 \\
\hline
\end{tabular}\label{TabConRateP10}
\end{table}

\begin{figure}[h]
\centering
\subfigure[]{\includegraphics[width=0.28\textwidth]{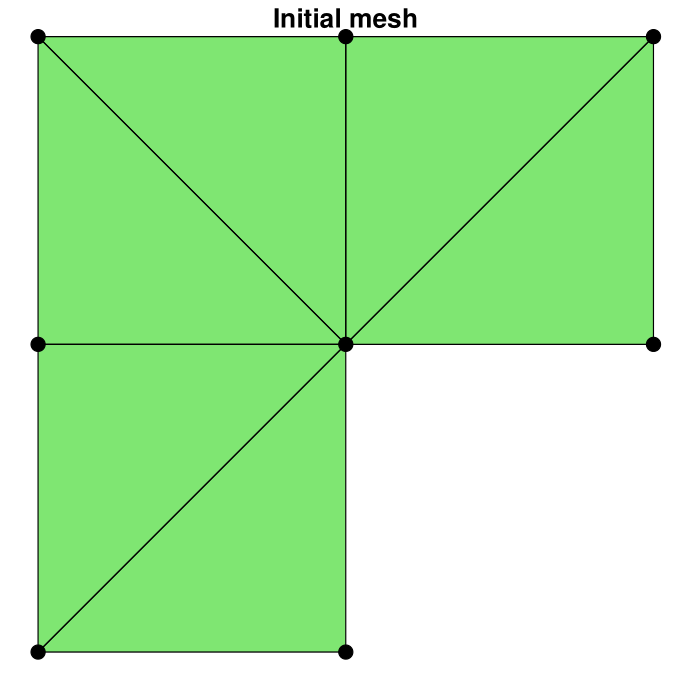}}
\subfigure[]{\includegraphics[width=0.27\textwidth]{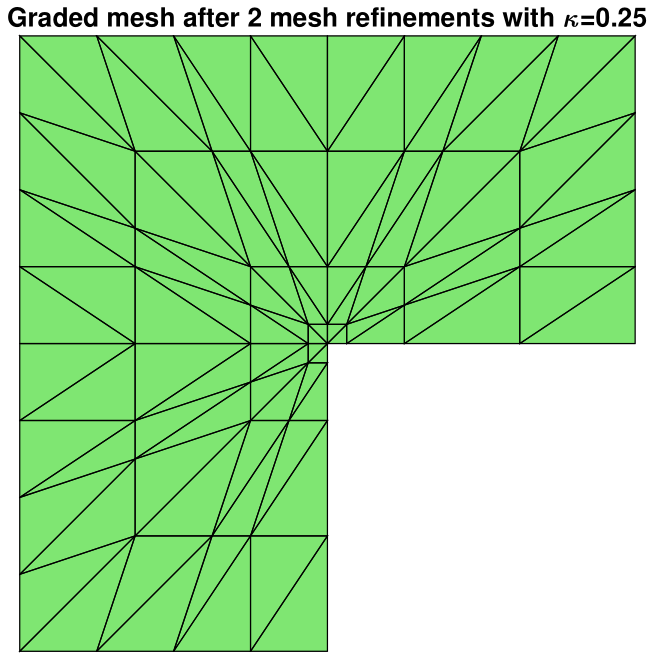}}
\subfigure[]{\includegraphics[width=0.32\textwidth]{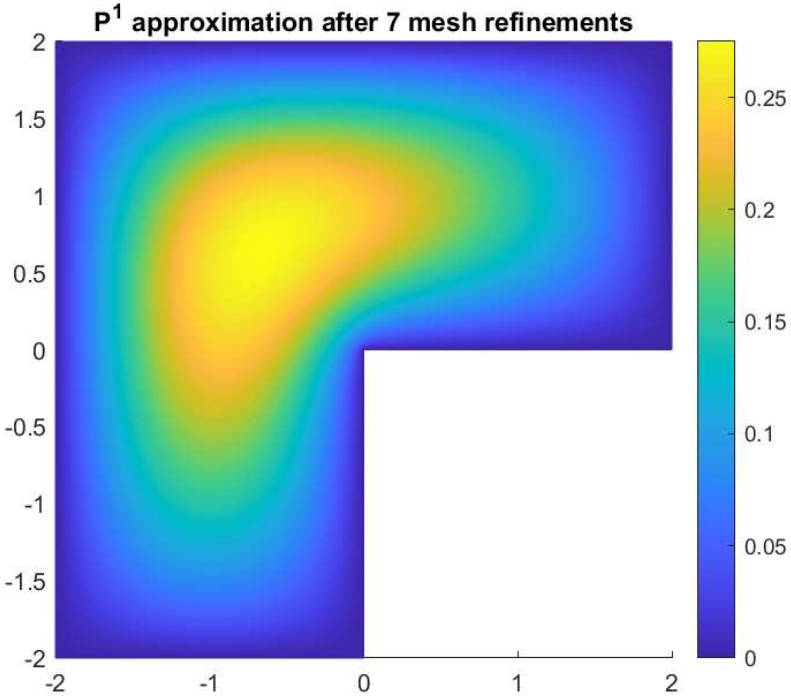}}
\subfigure[]{\includegraphics[width=0.30\textwidth]{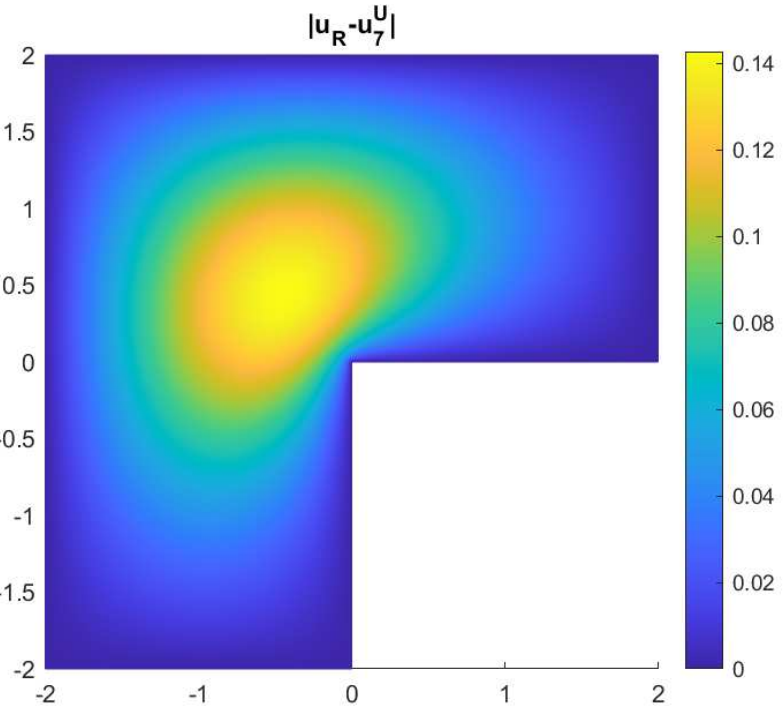}}
\subfigure[]{\includegraphics[width=0.30\textwidth]{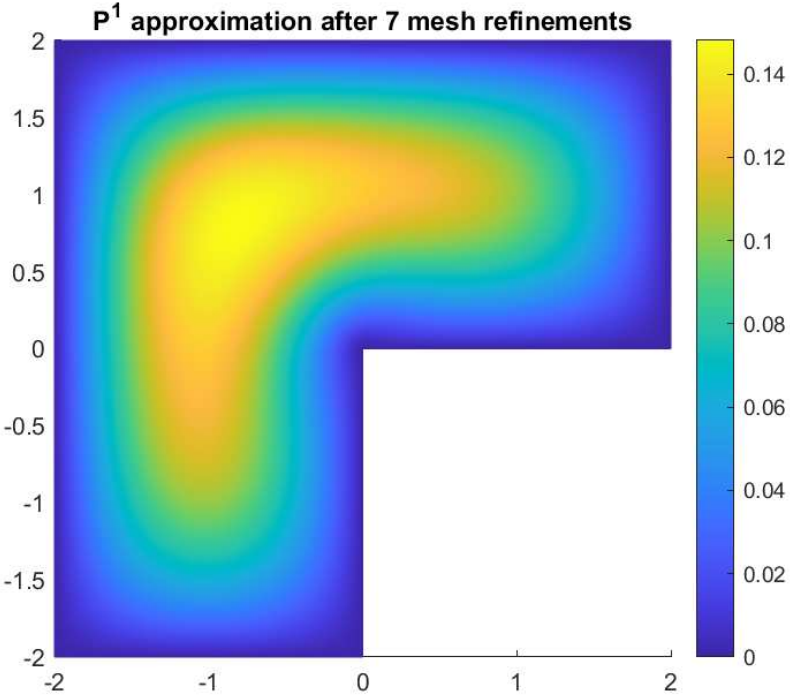}}
\subfigure[]{\includegraphics[width=0.30\textwidth]{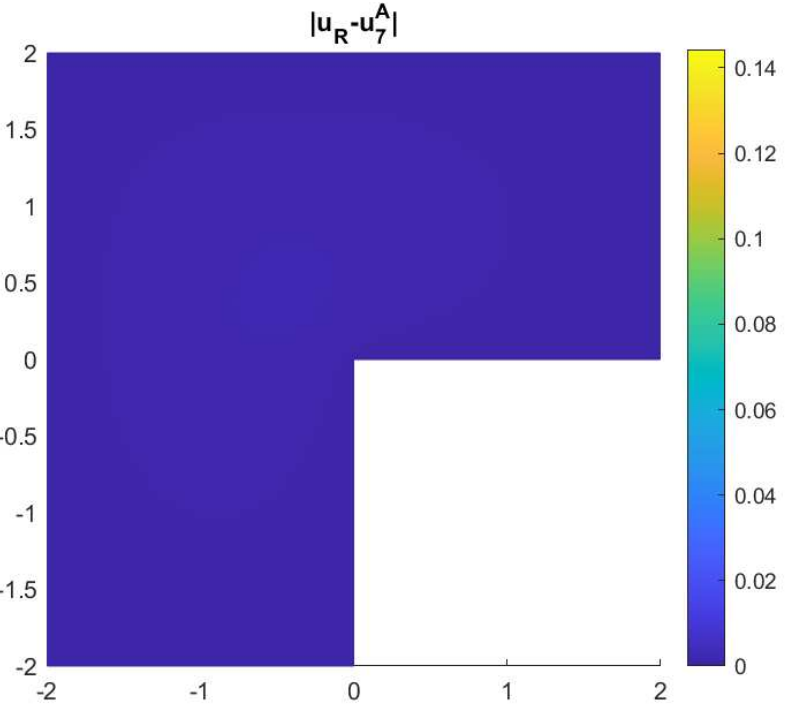}}
\caption{The L-shaped domain (Example \ref{P1h}):  (a) the initial mesh; (b) the graded mesh after two refinements; (c) the solution $u^U_7$ of the usual mixed method; (d) the difference $|u_R-u^U_7|$; (e) the solution $u^A_7$ from Algorithm \ref{femalg}; (f) the difference $|u_R-u^A_7|$.}\label{Mesh_Init}
\end{figure}
{
For the usual mixed finite element method, although the  solution $u^U_j$ does not converge to the solution of  (\ref{eqnbi}), we notice that both  $u^U_j$ and $w^U_j$, $j\geq 0$, are converging sequences.
The numerical convergence rate $\mathcal R$ (\ref{rate}) for  $u^U_j$ and $w^U_j$ on a sequence of graded meshes (including quasi-uniform meshes) is reported in Table \ref{TabConRateP10}. In the table, we observe that  $u^U_j$ and $w^U_j$ have similar convergence rates:  $\mathcal R<1$ on quasi-uniform meshes and on the graded meshes with $\kappa=0.4$; and $\mathcal R=1$ on graded meshes with $\kappa\leq 0.3$. These results indicate that the usual mixed finite element solution $u^U_j$ converges to the solution $\bar u$ of (\ref{eqn7}) in $H^1(\Omega)$. Recall however that $\bar u\neq u$ when the domain has reentrant corners.}

\begin{table}[!htbp]\tabcolsep0.03in
\caption{$H^1$ convergence history   in the L-shaped domain on graded meshes.}
\begin{tabular}[c]{|c|ccccc|ccccc|}
\hline
\multirow{2}{*}{$\kappa\backslash j$}  & \multicolumn{5}{|c|}{$u_j^A$} & \multicolumn{5}{|c|}{$w_j^A$}  \\
\cline{2-11}
& $j=5$ & $j=6$ & $j=7$ & $j=8$ & $j=9$ & $j=5$ & $j=6$ & $j=7$ & $j=8$ & $j=9$\\
\hline
$\kappa=0.1$ & 0.95 & 0.98 & 0.99 & 1.00 & 1.00 & 0.95 & 0.98 & 0.99 & 1.00 & 1.00 \\
\hline
$\kappa=0.2$ & 0.95 & 0.99 & 1.00 & 1.00 & 1.00 & 0.96 & 0.99 & 0.99  & 1.00 & 1.00\\
\hline
$\kappa=0.3$ & 0.97 & 0.99 & 1.00 & 1.00 & 1.00 & 0.96 & 0.98 & 0.99 & 0.99  & 0.99 \\
\hline
$\kappa=0.4$ & 0.98 & 0.99 & 1.00 & 1.00 & 1.00 & 0.94 & 0.95 & 0.95  & 0.94 & 0.94 \\
\hline
$\kappa=0.5$ & 0.97 & 0.99 & 0.99 & 1.00 & 1.00 & 0.87 & 0.84 & 0.80  & 0.77 & 0.74 \\
\hline
\end{tabular}\label{TabConRateP1}
\end{table}

In the {last} test,
we exam the convergence rates of  the finite element solution $u^A_j$ and the auxiliary finite element solution $w^A_j$ from Algorithm \ref{femalg} on a sequence of graded meshes (including quasi-uniform meshes).
The $H^1$ convergence rates (\ref{rate}) for the finite element solutions $u^A_j$ and $w^A_j$ are reported in Table \ref{TabConRateP1}.
For  $u_j^A$,  the optimal convergence rate ($\mathcal R=1$) is achieved on all  meshes with  $\kappa \in (0,0.5]$. For the auxiliary solution $w_j^A$, we observe that the convergence rate is not optimal on quasi-uniform meshes and on  the graded meshes with $\kappa=0.4$; and  the optimal convergence rate $\mathcal R=1$ is obtained on graded meshes when $\kappa\leq 0.3$. These numerical results justify the theory (Theorem \ref{thmuerr}, Theorem \ref{gthmuerr}, and Lemma \ref{glemwerrl2}) developed early in this paper. Namely, the numerical solution $u^A_j$ converges to $u$ in the optimal rate on quasi-uniform meshes and on graded meshes, while $w^A_j$ shall converge to $w$ in the optimal rate when $\kappa<2^{-\frac{\omega}{\pi}}=2^{-\frac{3}{2}}\approx 0.354$. For $\kappa>2^{-\frac{3}{2}}$ ($\kappa=0.4, 0.5$ in Table \ref{TabConRateP1}), $w^A_j$ shall converge to $w$ in a reduced rate due to the fact that $w$ is singular near the reentrant corner ($w\notin H^2(\Omega)$ and $w\in H^{\alpha}(\Omega)$ for $\alpha<1+\frac{2}{3}\approx 1.667$ see  (\ref{wh1err})).

\end{example}

\section*{Acknowledgments}
H. Li was supported in part by the National Science Foundation Grant DMS-1819041 and by the Wayne State University Faculty Competition for Postdoctoral Fellows Award. Z. Zhang was supported in part by the National Natural Science Foundation of China grants NSFC 11871092 and NASF U1930402.

\bibliography{LYZ20}

\bibliographystyle{plain}

\end{document}